\documentclass[12pt]{amsart}

\usepackage{amsfonts,latexsym,rawfonts,amsmath,amssymb,amsthm,mathrsfs}
\usepackage[colorlinks=true,urlcolor=blue,linkcolor=blue,anchorcolor=blue,citecolor=blue]{hyperref}
\usepackage{graphicx}
\usepackage{comment}
\usepackage[]{caption2}

\usepackage[pagewise]{lineno}
\usepackage{amsthm}
\usepackage{lipsum}
\usepackage[raggedright]{sidecap}\sidecaptionvpos{figure}{t}
\usepackage{tikz}
\usepackage{amssymb}
\usepackage{enumerate}
\newtheorem{mainthm}{Theorem}

\makeatletter
\@namedef{subjclassname@2020}{\textup{2020} Mathematics Subject Classification}
\makeatother

\numberwithin{equation}{section}

\RequirePackage{color}
\theoremstyle{plain}

\newtheorem{theorem}{Theorem}[section]
\newtheorem{lemma}{Lemma}[section]
\newtheorem{corollary}{Corollary}[section]
\newtheorem{proposition}{Proposition}[section]

\newtheorem{remark}{Remark}[section]
\def\R{\mathbb R}

\begin{document}
\title[On rigidity of hypersurfaces with constant shifted curvature functions]{On rigidity of hypersurfaces with constant shifted curvature functions in hyperbolic space}

\author{Weimin Sheng}
\address{Weimin Sheng: School of Mathematical Sciences, Zhejiang University, Hangzhou 310058, China.}
\email{weimins@zju.edu.cn}

\author{Yinhang Wang}
\address{Yinhang Wang: School of Mathematical Sciences, Zhejiang University, Hangzhou 310058, China.}
\email{22035021@zju.edu.cn}

\author{Jie Wu}
\address{Jie Wu: School of Mathematical Sciences, Zhejiang University, Hangzhou 310058, China.}
\email{wujiewj@zju.edu.cn}


\begin{abstract}
In this paper, we first give some new characterizations of geodesic spheres in the hyperbolic space by the condition that hypersurface has constant weighted shifted mean curvatures, or constant weighted shifted mean curvature ratio, which generalize the result of Hu-Wei-Zhou \cite{HWZ23}. Secondly, we investigate several rigidity problems for hypersurfaces in the hyperbolic space with constant linear combinations of weighted shifted mean curvatures as well as radially symmetric shifted mean curvatures. As applications, we obtain the rigidity results for hypersurfaces with constant linear combinations of mean curvatures in a general form and constant Gauss-Bonnet curvature $L_k$ under weaker conditions, which extend the work of the third author and Xia \cite{WX14}.
\end{abstract}

\maketitle

\baselineskip16pt
\parskip3pt

\section{Introduction}
The rigidity problem of hypersurfaces with constant curvature functions is a fundamental question in differential geometry. A classical theorem due to Alexandrov \cite{A56} states that any closed, embedded hypersurface in Euclidean space with constant mean curvature is a round sphere. Alexandrov's method is based on the maximum principle for elliptic equations and is now referred to as Alexandrov's reflection method. This result is remarkable in that it requires no assumptions about the topology of the hypersurface, which improves previous results due to S\"{u}ss \cite{S52} and Hsiung \cite{H54}. Later, Reilly \cite{R77} provided a new proof of the Alexandrov's theorem by using his famous integral formula. Following the work of Reilly \cite{R77}, Ros \cite{Ros88} generalized Alexandrov's result to the hypersurfaces with constant scalar curvature. Also, by using the classical Alexandrov's reflection method, Korevaar \cite{K88} gave another proof to this result and indicated that Alexandrov's reflection method works as well for hypersurfaces in the hyperbolic space and the hemisphere. Later, Ros \cite{Ros87} extended his result to any constant $k$-mean curvature, which was also proved by Montiel and Ros \cite{MR91} using a direct integral method due to Heintze-Karcher \cite{HK78}. The argument in Montiel and Ros's paper \cite{MR91} also applies to the hyperbolic space and the hemisphere.

After the work of Montiel and Ros, lots of extensions appeared on such rigidity topic. For example, Koh \cite{K98,K00} gave a new characterization of spheres in terms of the ratio of two mean curvatures. Aledo-Al\'{\i}as-Romero \cite{A99} extended the result to compact space-like hypersurfaces with constant higher order mean curvature in de Sitter space. In \cite{HLMG09}, He-Li-Ma-Ge investigated the compact embedded hypersurfaces with constant higher order anisotropic mean curvatures. In \cite{B13}, Brendle showed that Alexandrov Theorem holds in general warped product manifolds, including the (Anti-)deSitter-Schwarzschild manifolds as a typical example. Brendle and Eichmair \cite{BE13} later extended Brendle's result to any closed, star-shaped convex hypersurface with constant higher order mean curvature. For other generalizations, see for instance \cite{ADM06,ADM13,AD14,BC97,HMZ01,LWX14,M99} and references therein.

In a different direction, the curvature quantity with weight $V$ appears naturally. Here $V=\cosh r$, where $r$ is the hyperbolic distance to a fixed point in $\mathbb{H}^{n+1}$. For instance, the weighted mean curvature integral $\int_{\Sigma} V H_1 d \mu$ appears naturally in the definition of the quasi-local mass in $\mathbb{H}^{n+1}$ and the Penrose inequality for asymptotically hyperbolic graphs \cite{MRA13}. The Alexandrov-Fenchel inequalities with weight $V$ also hold in the hyperbolic space. For example, Brendle-Hung-Wang \cite{BHW16} and de Lima-Gir\~{a}o \cite{DG16} proved an Alexandrov-Fenchel-type inequality for the weighted mean curvature integral $\int_{\Sigma} V H_1 d \mu$. Ge, Wang and the third author \cite{GWW15} established an optimal inequality concerning $\int_{\Sigma} V H_k d \mu$. See also \cite{HLW22,SX19}. In \cite{W16}, the third author considered the weighted higher-order mean curvature $V H_k$ and gave a new characterization of geodesic spheres in the hyperbolic space $\mathbb{H}^{n+1}$.
\begin{mainthm}[\cite{W16}]\label{wj}
	Let $\Sigma$ be a closed embedded hypersurface in $\mathbb{H}^{n+1}$. If either of the following conditions holds on $\Sigma$:
	\begin{itemize}
		\item[(i)] $V H_k$ is constant for some $1\leq k\leq n$,
		\item[(ii)] $V \frac{H_k}{H_l}$ is constant for some $0\leq l< k\leq n$ and $H_l$ dose not vanish on $\Sigma$,
	\end{itemize}
	then $\Sigma$ is a centered geodesic sphere.
\end{mainthm}
The proof of Theorem \ref{wj} used the similar argument of Montiel and Ros as in \cite{MR91,Ros87}, and the Heintze-Karcher-type inequality for mean convex hypersurface in $\mathbb{H}^{n+1}$ (due to Brendle) also plays an important role. Recently, Hu-Wei-Zhou \cite{HWZ23} established a new Heintze-Karcher-type inequality for hypersurfaces with mean curvature $H>n$ in the hyperbolic space. See (\ref{hkf}) for the precise statement. As applications, they obtained an Alexandrov type theorem for closed embedded hypersurfaces with constant shifted $k$th mean curvature in hyperbolic space. For this, we recall the definition of shifted $k$th mean curvature. Let $\Sigma$ be a smooth hypersurface in $\mathbb{H}^{n+1}$, its shifted principal curvatures are defined by $\tilde{\kappa}=(\tilde{\kappa}_1, \cdots, \tilde{\kappa}_n)=(\kappa_1 -1, \cdots, \kappa_n -1)$,  where $\kappa=(\kappa_1, \cdots, \kappa_n)$ are the principal curvatures of $\Sigma$. Then the shifted $k$th mean curvature $H_k(\tilde{\kappa})$ of $\Sigma$ is given by the
the normalized $k$th elementary symmetric function of  $\tilde{\kappa}$. These definitions arise quite naturally in the setting of horospherically convex (h-convex) geometry of hypersurfaces in $\mathbb{H}^{n+1}$ (see \cite{BCW21,EGL09} for instance).
\begin{mainthm}[\cite{HWZ23}]\label{hwz}
	Let $\Sigma$ be a closed embedded hypersurface in $\mathbb{H}^{n+1}$. If the shifted $k$th mean curvature $H_k(\tilde{\kappa})$ is constant for some $k \in \{1,\cdots,n \}$, then $\Sigma$ is a geodesic sphere.
\end{mainthm}
Firstly, we give some generalizations of the above result for closed hypersurfaces in the hyperbolic space. We prove the following theorem.
\begin{theorem}\label{thm1.1}
	Assume that $\chi(s)$ is a smooth, positive and monotone non-decreasing function defined on $\mathbb{R}^{+}$. Let $\Sigma$ be a closed embedded hypersurface in $\mathbb{H}^{n+1}$. If either of the following conditions holds on $\Sigma$:
	\begin{itemize}
		\item[(i)] $\chi(V) H_k(\tilde{\kappa})$ is constant for some $1\leq k\leq n$,
		\item[(ii)] $\chi(V) \frac{H_k(\tilde{\kappa})}{H_l(\tilde{\kappa})}$ is constant for some $0\leq l< k\leq n$ and $H_l(\tilde{\kappa})$ dose not vanish on $\Sigma$,
	\end{itemize}
	then $\Sigma$ is a geodesic sphere. Moreover, if $\chi$ is strictly increasing, then $\Sigma$ is a centered geodesic sphere.
\end{theorem}
	If $\chi=1$, case $(i)$ of Theorem \ref{thm1.1} reduces to Hu-Wei-Zhou's result \cite{HWZ23}.
Our second result is the rigidity for hypersurfaces with constant linear combinations of weighted shifted mean curvatures in $\mathbb{H}^{n+1}$.
\begin{theorem}\label{thm1.3}
	Let $n\geq2$. Assume that $\chi(s)$ is a smooth, positive and monotone non-decreasing function defined on $\mathbb{R}^{+}$. Let $0\leq k\leq n$ be an integer and $\Sigma$ be a closed hypersurface in $\mathbb{H}^{n+1}$ with $\tilde{\kappa}\in\overline{\Gamma_{k}^+}$. If one of the following holds:
	\begin{itemize}
		\item[(i)] $2\leq l<k\leq n$ and there are nonnegative constants $\{a_i\}_{i=1}^{l-1}$ and $\{b_j\}_{j=l}^{k}$, at least one of them not vanishing, such that
		\begin{equation*}
			\sum_{i=1}^{l-1}a_i H_i (\tilde{\kappa})=\sum_{j=l}^{k}b_j (\chi(V)H_j (\tilde{\kappa}));
		\end{equation*}
		\item[(ii)] there are nonnegative constants $a_0$ and $\{b_j\}_{j=1}^{k}$, at least one of them not vanishing, such that
		\begin{equation*}
			a_0 =\sum_{j=1}^{k}b_j (\chi(V)H_j (\tilde{\kappa}));
		\end{equation*}
		\item[(iii)] $\Sigma$ is star-shaped, $1\leq l<k\leq n-1$ and there are nonnegative constants $\{a_i\}_{i=0}^{l-1}$ and $\{b_j\}_{j=l}^{k}$, at least one of them not vanishing, such that
		\begin{equation*}
			\sum_{i=0}^{l-1}a_i H_i (\tilde{\kappa})=\sum_{j=l}^{k}b_j (\chi(V)H_j (\tilde{\kappa})),
		\end{equation*}
	\end{itemize}
	then $\Sigma$ is a geodesic sphere.
\end{theorem}

\begin{remark}
	Theorem \ref{thm1.3} contains the case that the weighted shifted mean curvature ratio $\chi(V)\frac{H_{k}(\tilde{\kappa})}{H_{l}(\tilde{\kappa})}$ are constant for $k>l$. We notice that the condition of $\tilde{\kappa}\in\overline{\Gamma_{k}^+}$ is superfluous in this case since it is implied by the constancy of $\chi(V)\frac{H_{k}(\tilde{\kappa})}{H_{l}(\tilde{\kappa})}$. Furthermore, we also prove a similar rigidity result with weight $\chi(V-u)$ for h-convex hypersurfaces. See Theorem \ref{cvu} below.
\end{remark}
Theorem \ref{thm1.3} will be proved by using the classical integral method due to Hsiung \cite{H54} and Reilly \cite{R77}. The main tools are a family of Newton-Maclaurin inequalities as well as a New Minkowski type formula, Lemma \ref{wmkl}.

Following the idea of Theorem \ref{thm1.3}, we investigate the following general form of rigidity result for linear combinations of shifted higher order mean curvatures.
\begin{theorem}\label{thm+1}
	Let $n\geq2$. Let $0\leq k\leq n$ be an integer and $\Sigma$ be a closed hypersurface in $\mathbb{H}^{n+1}$ with $\tilde{\kappa}\in\overline{\Gamma_{k}^+}$. Let $r$ be the distance in $\mathbb{H}^{n+1}$ from a fixed point $p_{0}$. If one of the following holds:
	\begin{itemize}
		\item[(i)] $2\leq l<k\leq n$ and there are two families of nonnegative smooth functions $\{a_i (r)\}_{i=1}^{l-1}$ and $\{b_j (r)\}_{j=l}^{k}$, which are monotone decreasing  and monotone increasing respectively, at least one of them not vanishing, such that
		\begin{equation*}
			\sum_{i=1}^{l-1}a_i (r) H_i (\tilde{\kappa})=\sum_{j=l}^{k}b_j (r) H_j (\tilde{\kappa});
		\end{equation*}
		\item[(ii)] there are nonnegative smooth monotone decreasing functions $a_0 (r)$ and monotone increasing functions $\{b_j (r)\}_{j=1}^{k}$, at least one of them not vanishing, such that
		\begin{equation}\label{a0r}
			a_0 (r)=\sum_{j=1}^{k}b_j (r) H_j (\tilde{\kappa});
		\end{equation}
		\item[(iii)] $\Sigma$ is star-shaped, $1\leq l<k\leq n-1$ and there are two families of nonnegative smooth functions $\{a_i (r)\}_{i=0}^{l-1}$ and $\{b_j (r)\}_{j=l}^{k}$, which are monotone decreasing  and monotone increasing respectively, at least one of them not vanishing, such that
		\begin{equation*}
			\sum_{i=0}^{l-1}a_i (r) H_i (\tilde{\kappa})=\sum_{j=l}^{k}b_j (r) H_j (\tilde{\kappa}),
		\end{equation*}
	\end{itemize}
	then $\Sigma$ is a geodesic sphere.
\end{theorem}
In fact, the above equation is the equation of Krylov type which has been introduced and studied by Krylov in \cite{K1995} in the following form
\begin{equation}\label{Krylov problem}
	\sum_{i=0}^{k-1} \alpha_{i}(x) H_i\left(D^2 u\right)=H_k\left(D^2 u\right), \quad x \in \Omega \subset \mathbb{R}^n,
\end{equation}
where $\Omega $ is a $(k-1)$-convex domain and $u$ is a twice continuously differentiable function. It's an extension of the work   investigated by Caffarelli et al. \cite{CNS1984, CNS1985} on the Hessian equation. Krylov observed that if $\alpha_{i}(x) \geq 0$ for $0 \leq i \leq k-1$, the natural admissible cone to make the equation \eqref{Krylov problem} elliptic is the $\Gamma_k^{+}$-cone which is the same as the $k$-Hessian equation case.

Theorem \ref{thm+1} contains the case where $\frac{H_{k} (\tilde{\kappa})}{H_{l} (\tilde{\kappa})}=\eta (r)$ for some monotone decreasing function $\eta$ and $k>l$. This result is the general form of Theorem \ref{thm1.1}(ii).

Next, we investigate a rigidity result of non-linear form.
\begin{theorem}\label{thm+2}
	Let $n\geq2$. Let $\Sigma$ be a closed star-shaped hypersurface with $\tilde{\kappa}\in\overline{\Gamma_{k}^+}$ and $r$ be the distance in $\mathbb{H}^{n+1}$ from a fixed point $p_{0}$. If there are two families of nonnegative, smooth,
monotone increasing functions $\{a_j (r)\}_{j=1}^{k}$ and $\{b_j (r)\}_{j=1}^{k}$, such that
	\begin{equation}\label{bjhj}
		\sum_{j=1}^{k}\bigg(a_j (r) H_j (\tilde{\kappa}) + b_j (r) H_1 (\tilde{\kappa}) H_{j-1} (\tilde{\kappa})\bigg) = \eta (r),
	\end{equation}
	for some smooth positive radially symmetric function $\eta (r)$ which is monotone decreasing in $r$, then $\Sigma$ is a geodesic sphere.
\end{theorem}
Theorem \ref{thm+2} contains two special cases which are worth mentioning: (i) $H_{k}(\tilde{\kappa})=\eta(r)$ and (ii) $H_{1} (\tilde{\kappa}) H_{k-1} (\tilde{\kappa}) = \eta (r)$, where $\eta(r)$ is a monotone decreasing function. The second case can be seen as a ``non-linear" version of Theorem \ref{hwz}.

In \cite[Theorem 3]{KLP18}, Kwong, Lee and Pyo proved a rigidity theorem for self-expanding solitons to the weighted generalized inverse curvature flow in $\R^{n+1}$
\begin{equation}
	\frac{d}{dt}X=\sum_{0 \leq i < j \leq n} a_{i,j}\left(\frac{H_{i}}{H_{j}}\right)^{\frac{1}{j-i}}\nu,
\end{equation}
where the weight functions $\left\{a_{i,j}(x)|0 \leq i < j \leq n\right\}$ are non-negative functions on the hypersurface satisfying $\sum_{0 \leq i < j \leq n}a_{i,j}(x)=1$. We next extend it to the hyperbolic case.
\begin{theorem}\label{thm+3}
	Let $\left\{a_{i,j}(x)|0 \leq i < j \leq n\right\}$ be non-negative functions on the hypersurface satisfying $\sum_{0 \leq i < j \leq n}a_{i,j}(x)=1$, $k=\max\left\{j|a_{i,j}>0 \text{~for some~} 0 \leq i < j \leq n\right\}$, and $\Sigma$ be a closed hypersurface with $\tilde{\kappa}\in\Gamma_{k}^+$ in $\mathbb{H}^{n+1}$. If there exists a constant $\beta > 0$ satisfying
	\begin{equation}\label{aij}
		\sum_{0 \leq i < j \leq k} a_{i,j}\left(\frac{H_{i}(\tilde{\kappa})}{H_{j}(\tilde{\kappa})}\right)^{\frac{1}{j-i}} = \beta \frac{u}{V-u},
	\end{equation}
	where $u=\left\langle \sinh r\partial_{r}, \nu \right\rangle$ is the support function, then $\Sigma$ is a geodesic sphere.
	
\end{theorem}

For the first application of Theorem \ref{thm+1} is a large class of the rigidity results for hypersurfaces with constant linear combinations of $H_{k}$.
\begin{corollary}\label{coro1}
	 Let $1\leq k\leq n$ be an integer and $\Sigma$ be a closed hypersurface in $\mathbb{H}^{n+1}$ with $\tilde{\kappa}\in\overline{\Gamma_{k}^+}$. If there are nonnegative constants $a_0$ and $\{b_j\}_{j=1}^{k}$, at least one of them not vanishing, such that
		\begin{equation*}
			a_0 + \sum_{i=1}^{k}(-1)^{i+1}b_{i}=\sum_{l=1}^{k} \sum_{j=0}^{k-l}(-1)^{j}\binom{l+j}{l}b_{l+j} H_l (\kappa),
		\end{equation*}
	then $\Sigma$ is a geodesic sphere.
\end{corollary}
\begin{remark}
	The third author and Xia \cite{WX14} proved the rigidity for $k$-convex hypersurfaces with constant linear combinations of $H_{k}$, i.e.,
	$$
	a_0 = \sum_{j=1}^{k}b_j H_j,
	$$
	where the coefficients $a_0$ and $b_j$ are nonnegative. In fact, Corollary \ref{coro1} obtains a larger class of such rigidity results including linear combinations of $H_{k}$ with negative coefficients.
\end{remark}
The second application of Theorem \ref{thm+1} is about rigidity problems on some intrinsic curvature functions of induced metric from that of the hyperbolic space. Let $(\Sigma,g)$ be a hypersurface in $\mathbb{H}^{n+1}$. The $k$th Gauss-Bonnet curvature $L_{k}$ of the metric $g$ is defined by
\begin{equation}\label{lk}
	L_k(g):=\frac{1}{2^k} \delta_{j_1 j_2 \cdots j_{2 k-1} j_{2 k}}^{i_1 i_2 \cdots i_{2 k-1} i_{2 k}} R_{i_1 i_2}{ }^{j_1 j_2} \cdots R_{i_{2 k-1} i_{2 k}}{ }^{j_{2 k-1} j_{2 k}},
\end{equation}
where $R_{i j}{ }^{k l}$ is the Riemannian curvature tensor in the local coordinates with respect to the metric $g$, and the generalized Kronecker delta is defined by
$$
\delta_{j_1 j_2 \cdots j_r}^{i_1 i_2 \cdots i_r}=\operatorname{det}\left(\begin{array}{cccc}
	\delta_{j_1}^{i_1} & \delta_{j_2}^{i_1} & \cdots & \delta_{j_r}^{i_1} \\
	\delta_{j_1}^{i_2} & \delta_{j_2}^{i_2} & \cdots & \delta_{j_r}^{i_2} \\
	\vdots & \vdots & \vdots & \vdots \\
	\delta_{j_1}^{i_r} & \delta_{j_2}^{i_r} & \cdots & \delta_{j_r}^{i_r}
\end{array}\right) .
$$
The $k$th Gauss-Bonnet curvature $L_k$ of the induced metric of a hypersurface in hyperbolic space can be expressed in terms of the shifted $k$th mean curvatures (see Lemma \ref{lkl} below). Explicitly,
$$
L_{k} = \binom{n}{2k} (2k)! \sum_{j=0}^{k}2^{j}\binom{k}{j}H_{2k-j}(\tilde{\kappa}).
$$
Notice here all the coefficients are positive. Therefore, as a direct consequence of Theorem \ref{thm+1}(ii), we have the following
\begin{corollary}\label{coro2}
	Let $1\leq k\leq\frac{n}{2}$ be an integer and $\Sigma$ be a closed hypersurface embedded in the hyperbolic space $\mathbb{H}^{n+1}$ with $\tilde{\kappa}\in\overline{\Gamma_{2k}^+}$. If there are nonnegative constants $a_0$ and $\{b_j\}_{j=1}^{k}$, at least one of them not vanishing, such that
	\begin{equation*}
		a_0=\sum_{j=1}^{k}b_j L_j,
	\end{equation*}
	then $\Sigma$ is a geodesic sphere. In particular, if $L_{k}$ is constant, then $\Sigma$ is a geodesic sphere.
\end{corollary}
\begin{remark}
	In \cite{WX14}, the third author and Xia proved the rigidity result on $L_{k}$, for horoconvex hypersurfaces. Here, a hypersurface in $\mathbb{H}^{n+1}$ is horospherical convex if $\tilde{\kappa}\in\overline{\Gamma_{n}^+}$. Corollary \ref{coro2} generalizes their result to the hypersurfaces with $\tilde{\kappa}\in\overline{\Gamma_{2k}^+}$.
\end{remark}

The rest of this paper is organized as follows. In Section \ref{sec2}, we first collect some basic facts about the $k$th normalized elementary symmetric function $H_{k}$. Then we provide some Minkowski-type formulas and inequalities, which are the most important tools of this paper. In Section \ref{sec3} and \ref{sec4}, we give the proof of  the main results, Theorem \ref{thm1.1} and Theorem  \ref{thm1.3}. Section \ref{sec+} is devoted to prove the general form of the rigidity results, Theorems \ref{thm+1}--\ref{thm+3}. In Section \ref{sec5}, we focus on the rigidity problem on the intrinsic Gauss-Bonnet curvatures and show Corollaries \ref{coro1} and \ref{coro2}.

\section{Preliminaries}\label{sec2}
In this section, first let us recall some basic definitions and properties of higher order mean curvature.

Let $\sigma_{k}$ be the $k$th elementary symmetry function $\sigma_{k}:\mathbb{R}^{n}\rightarrow\mathbb{R}$ defined by
$$
\sigma_k(\lambda)=\sum_{i_1<\cdots<i_k} \lambda_{i_1} \cdots \lambda_{i_k} \text { for } \lambda=\left(\lambda_1, \cdots, \lambda_{n}\right) \in \mathbb{R}^{n}.
$$
For a symmetric $n \times n$ matrix $A=\left(A_i^j\right)$, we set
\begin{equation}\label{sgmk}
	\sigma_k(A)=\frac{1}{k !} \delta_{j_1 \cdots j_m}^{i_1 \cdots i_m} A_{i_1}^{j_1} \cdots A_{i_m}^{j_m},
\end{equation}
where $\delta_{j_1 \cdots j_m}^{i_1 \cdots i_m}$ is the generalized Kronecker symbol.
If $\lambda(A)=\left(\lambda_1(A), \cdots, \lambda_n(A)\right)$ are the real eigenvalues of $A$, then
$$
\sigma_k(A)=\sigma_k(\lambda(A)) .
$$
The $k$th Newton transformation is defined as follows
$$
\left(T_k\right)_i^j(A)=\frac{\partial \sigma_{k+1}}{\partial A_j^i}(A)=\frac{1}{k !}\delta_{i i_1 \cdots i_{k}}^{j j_1 \cdots j_{k}} A_{j_1}^{i_1} \cdots A_{j_{k}}^{i_{k}}.
$$
If $A$ is a diagonal matrix with $\lambda(A)=\left(\lambda_1, \cdots, \lambda_n\right)$, then
$$
\left(T_k\right)_i^j(A) = \frac{\partial \sigma_{k+1}}{\partial \lambda_i}(\lambda)\delta_{i}^{j}.
$$
The following formulae for the elementary symmetric functions are well-known.
\begin{lemma}[\cite{R74}]\label{tklemma}
	We have
	
	\begin{equation}\label{tk1}
		\sum_{i, j}\left(T_{k-1}\right)_{i}^{j}(A) A_{j}^{i}=k \sigma_{k}(A) .
	\end{equation}
	\begin{equation}\label{tk2}
		\sum_{i, j}\left(T_{k-1}\right)_{i}^{j}(A) \delta_{j}^{i}=(n+1-k) \sigma_{k-1}(A) .
	\end{equation}
\end{lemma}
The $k$th G\r{a}rding cone $\Gamma_k^{+}$ is defined by
\begin{equation}
	\Gamma_k^{+}=\left\{\lambda \in \mathbb{R}^n: \sigma_i(\lambda)>0,1 \leq i \leq k\right\} .
\end{equation}
And its closure is denoted by $\overline{\Gamma_{k}^+}$. A symmetric matrix $A$ is said to belong to $\Gamma_k^{+}$ if its eigenvalues $\lambda(A) \in \Gamma_k^{+}$.
Let
\begin{equation}
	H_k = \frac{\sigma_k}{C_{n}^k},
\end{equation}
be the normalized $k$th elementary symmetry function. As a convention, we take $H_0 = 1$, $H_{-1} = 0$. We have the following Newton-Maclaurin inequalities.
\begin{lemma}[\cite{G02}]
	 For $1 \leq l<k \leq n$ and $\lambda \in \overline{\Gamma_{k}^{+}}$, the following inequalities hold:
	\begin{equation}\label{nm1}
		H_{k-1}(\lambda)H_{l}(\lambda) \geq H_{k}(\lambda)H_{l-1}(\lambda).
	\end{equation}
	\begin{equation}\label{nm2}
		H_{l}(\lambda) \geq H_{k}(\lambda)^{\frac{l}{k}}.
	\end{equation}
	Moreover, equality holds in (\ref{nm1}) or (\ref{nm2}) at $\lambda$ if and only if $\lambda=c(1,1, \cdots, 1)$.
\end{lemma}

Next, we collect some well-known results on geometry of hypersurfaces in hyperbolic space.

In this paper, we view hyperbolic space as the warped product space $\mathbb{H}^{n+1}=[0, \infty) \times \mathbb{S}^n$ equipped with the metric
$$
\bar{g}=d r^2+\lambda(r)^2 g_{\mathbb{S}^n},
$$
where $g_{\mathbb{S}^n}$ is the standard round metric of $\mathbb{S}^n$ and $\lambda(r)=\sinh r$. Let $\Sigma$ be a smooth hypersurface in $\mathbb{H}^{n+1}$. Denote $\bar{\nabla}$ and $\nabla$ as the Levi-Civita connection on $\mathbb{H}^{n+1}$ and $\Sigma$, respectively. Let $\{e_i\}_{i=1}^{n}$ and $\nu$ be an orthonormal basis and the unit outward normal of $\Sigma$, respectively. Then the induced metric $g$ of $\Sigma$ is $g_{i j}=\bar{g}\left(e_i, e_j\right)$, and the second fundamental form $h=\left(h_{i j}\right)$ of $\Sigma$ in $\mathbb{H}^{n+1}$ is given by
$$
h_{i j}=h\left(e_i, e_j\right)=\bar{g}\left(\bar{\nabla}_{e_i} \nu, e_j\right) .
$$
The principal curvatures $\kappa=\left(\kappa_1, \cdots, \kappa_n\right)$ of $\Sigma$ are the eigenvalues of the Weingarten matrix $\left(h_i^j\right)=\left(g^{j k} h_{k i}\right)$, where $\left(g^{i j}\right)$ is the inverse matrix of $\left(g_{i j}\right)$. The mean curvature of $\Sigma$ is defined as
$$
H=g^{i j}h_{i j}=\sum_{i=1}^{n}\kappa_i.
$$
A hypersurface $\Sigma$ in $\mathbb{H}^{n+1}$ is called star-shaped if its support function
\begin{equation}\label{spf}
	u=\left\langle \lambda(r)\partial_{r}, \nu \right\rangle \geq 0.
\end{equation}
Note that $X= \lambda(r) \partial_r$ is a conformal vector field satisfying
\begin{equation}\label{cvf}
	\bar{\nabla}X=\lambda' \bar{g}.
\end{equation}

We first show that there exists at least one elliptic point such that all principal curvatures $\kappa_i >1$ for $i=1, \cdots, n$ on any closed hypersurface in the hyperbolic space $\mathbb{H}^{n+1}$, by following the argument as Lemma 2.1 in \cite{LWX14}.
\begin{lemma}\label{elliptic}
	Let $\Sigma$ be a closed hypersurface in the hyperbolic space $\mathbb{H}^{n+1}$. Then there exists at least one elliptic point $x$ such that all principal curvatures $\kappa_i >1$ for $i=1, \cdots, n$ on $\Sigma$.
\end{lemma}
\begin{proof}
    Let $\left\{e_1, \cdots, e_{n}\right\}$ be a local orthonormal frame on $\Sigma$, and assume that the second fundamental form $h_{i j}=\left\langle\bar{\nabla}_{e_i} \nu, e_j\right\rangle$ is diagonal with eigenvalues $\kappa_1, \cdots, \kappa_{n}$. Then
	\begin{equation}\label{ninr}
		\nabla_{e_i} \nabla r=\nabla_{e_i}\left(\frac{1}{\lambda(r)} \lambda(r) \partial_r^{\top}\right)=-\frac{\lambda^{\prime}}{\lambda}\left(\nabla_{e_i} r\right) \partial_r^{\top}+\frac{1}{\lambda} \nabla_{e_i}\left(\lambda \partial_r^{\top}\right) .
	\end{equation}
	It follows from (\ref{cvf}) that
	\begin{equation}\label{nilrt}
		\begin{aligned}
			\nabla_{e_i}\left(\lambda \partial_r^{\top}\right) & =\nabla_{e_i}\left(\lambda \partial_r-\left\langle\lambda \partial_r, \nu\right\rangle \nu\right)=\left(\bar{\nabla}_{e_i}\left(\lambda \partial_r-\left\langle\lambda \partial_r, \nu\right\rangle \nu\right)\right)^{\top} \\
			& =\lambda^{\prime} e_i-\left\langle\lambda \partial_r, \nu\right\rangle \kappa_i e_i.
		\end{aligned}
	\end{equation}
	Substituting (\ref{nilrt}) into (\ref{ninr}), we get
	\begin{equation}\label{ninr2}
			\nabla_{e_i} \nabla r=-\frac{\lambda^{\prime}}{\lambda}\left(\nabla_{e_i} r\right) \partial_r^{\top}+\frac{1}{\lambda}\left(\lambda^{\prime}-\left\langle\lambda \partial_r, \nu\right\rangle \kappa_i\right) e_i.
	\end{equation}
	Now we consider at the maximum point $x$ of $r$, we have $\nabla r=0$, $\nu=\partial_r$ and $\nabla^2 r \leq 0$ at $x$. Then from (\ref{ninr2}), we get
	$$
	\kappa_i \geq \frac{\lambda^{\prime}}{\lambda}>1, \quad i=1, \cdots, n,
	$$
	i.e. $x$ is an elliptic point of $\Sigma$ such that all principal curvatures $\kappa_i >1$ for $i=1, \cdots, n$.
\end{proof}
We need the following Minkowski type formula in hyperbolic space, which is included in the proof of \cite[Lemma 2.6]{HLW22}  or \cite[Lemma 2.3]{HWZ23}. For completeness, we involve the proof here.
\begin{lemma}\label{mkl}
	Let $\Sigma$ be a closed hypersurface in $\mathbb{H}^{n+1}$. Denote by $V=\cosh r$ and $u=\left\langle \bar{\nabla}V, \nu \right\rangle=\left\langle \lambda\partial_{r}, \nu \right\rangle$. Then we have
	\begin{equation}\label{mkf}
		\int_{\Sigma}\left( V- u \right) H_{k-1}(\tilde{\kappa}) d \mu=\int_{\Sigma} u H_k(\tilde{\kappa}) d \mu, \quad k=1, \cdots, n.
	\end{equation}
\end{lemma}
\begin{proof}
	It follows from the Gauss-Weingarten formula and (\ref{cvf}) that
	\begin{equation}\label{hv}
		\nabla_{i}\nabla_{j} V=\left\langle \bar{\nabla}_{i}(\lambda \partial_r), e_j\right\rangle - uh_{ij}=Vg_{ij}- uh_{ij}.
	\end{equation}
Denote $\tilde h=(h_j^i- \delta_j^i)_{n\times n}$ and note that $H_k(\tilde\kappa)=H_k(\tilde h)$.
	Multiplying (\ref{hv}) by the $k$th Newton transform tensor $\left(T_{k-1}\right)_i^j(\tilde{h})$ and summing over $i$, $j$, we obtain
	\begin{equation}\label{tkv}
		\begin{aligned}
			\sum_{i, j=1}^{n}\left(T_{k-1}\right)_i^j(\tilde{h}) \nabla^i \nabla_j V & =\sum_{i, j=1}^{n}\left(T_{k-1}\right)_i^j(\tilde{h})\left(V \delta_j^i-u h_j^i\right) \\
			& =\sum_{i, j=1}^{n}\left(T_{k-1}\right)_i^j(\tilde{h})\left[\left(V-u\right) \delta_j^i-u\left(h_j^i- \delta_j^i\right)\right] \\
			& =kC_n^k \left(\left(V-u\right) H_{k-1}(\tilde{\kappa})-u H_k(\tilde{\kappa})\right) .
		\end{aligned}
	\end{equation}
	where in the last equality we used (\ref{tk1}) and (\ref{tk2}). Notice that $\tilde{h}_{i j}=h_{i j}- g_{i j}$ is a Codazzi tensor, i.e., $\nabla_{\ell} \tilde{h}_{i j}$ is symmetric in $i, j, \ell$. It follows that $\left(T_{k-1}\right)_i^j(\tilde{h})$ is divergence free. The equation (\ref{mkf}) follows from integration by parts and the divergence free property of $\left(T_{k-1}\right)_i^j(\tilde{h})$.
\end{proof}
For later purpose to prove the rigidity result on weighted shifted curvature functions, we need to extend the above lemma to the following type.
\begin{lemma}\label{wmkl}
	Assume that $\chi(s)$ is a smooth function defined on $\mathbb{R}^{+}$. Let $\Sigma$ be a closed hypersurface in $\mathbb{H}^{n+1}$. We have
	\begin{equation}\label{wmk}
		\begin{aligned}
			\int_{\Sigma} \chi(V) u H_k(\tilde{\kappa}) d \mu= & \int_{\Sigma} \chi(V)(V-u) H_{k-1}(\tilde{\kappa}) d \mu  +  \frac{1}{k C_{n}^k} \int_{\Sigma} \chi'(V)\left(T_{k-1}\right)_{i}^{j}(\tilde{h})\nabla^{i}V \nabla_j V d \mu ,
		\end{aligned}
	\end{equation}
	where $\tilde h=(h_j^i- \delta_j^i)_{n\times n}$.
	Moreover, if $\chi$ is non-decreasing and $\tilde{\kappa}\in \overline{\Gamma_{k}^{+}}$, then we have
	\begin{equation}\label{wmki}
		\int_{\Sigma}\chi(V) u H_k(\tilde{\kappa}) d \mu \geq \int_{\Sigma} \chi(V)(V-u) H_{k-1}(\tilde{\kappa}) d \mu.
	\end{equation}
\end{lemma}
\begin{proof}
	From (\ref{tkv}), we arrive at
	\begin{equation}
		\left(T_{k-1}\right)_i^j(\tilde{h}) \nabla^i \nabla_j V=kC_n^k \left(\left(V-u\right) H_{k-1}(\tilde{\kappa})-u H_k(\tilde{\kappa})\right).
	\end{equation}
	Multiplying above equation by the function $\chi(V)$ and integrating by parts, one obtains the desired result (\ref{wmk}). Under the assumption that $\tilde{\kappa}\in \overline{\Gamma_{k}^{+}}$, the $(k-1)$th Newton tensor $T_{k-1}$ is semi-positively definite (see e.g. Guan \cite{G02}), hence
	$$
	\left(T_{k-1}\right)_{i}^{j}(\tilde{h})\nabla^{i}V \nabla_j V \geq 0.
	$$
	Together with assumption $\chi' \geq 0$, (\ref{wmki}) holds.
\end{proof}
Finally, we need the Heintze-Karcher-type inequality due to Hu, Wei and Zhou \cite{HWZ23}.
\begin{proposition}[\cite{HWZ23}]\label{hkp}
	 Let $\Omega$ be a bounded domain with smooth boundary $\Sigma=\partial\Omega$ in hyperbolic space $\mathbb{H}^{n+1}$ $(n\geq 2)$. Fix a point $o\in \mathbb{H}^{n+1}$ and $V(x)=\cosh r(x)$, where $r(x)=d(o,x)$ is the distance to this point o. Assume that the mean curvature of $\Sigma=\partial \Omega$ satisfies $H>n$, then
	\begin{equation}\label{hkf}
		\int_{\Sigma} \frac{V-u}{H-n} d \mu \geq \frac{n+1}{n} \int_{\Omega} V d \mathrm{vol} ,
	\end{equation}
	where $u=\left\langle \bar{\nabla}V, \nu \right\rangle=\left\langle \sinh r\partial_{r}, \nu \right\rangle$ is the support function of $\Sigma$ and $\nu$ denotes the unit outward normal of $\Sigma$. Equality holds in (\ref{hkf}) if and only if $\Sigma$ is umbilic.
\end{proposition}
We remark that for the case of $n=1$ (the curve case), the inequality (\ref{hkf}) was proved by Li and Xu \cite{LX}.

\section{Proof of Theorem \ref{thm1.1}}\label{sec3}
After all the preparation work, we are ready to prove our main theorems. We start with the weighted shifted curvature and their quotients. The Minkowski-type inequalities and the Heintze-Karcher-type inequality play important roles.
\begin{proof}[Proof of Theorem \ref{thm1.1}]
(i) Fix $k \in\{1, \cdots, n\}$. Since $\chi(V)H_{k}(\tilde{\kappa})=c$ for some constant $c$, the condition $\chi > 0$ and Lemma \ref{elliptic} imply that $c > 0$, which in turn implies that $H_{k}(\tilde{\kappa}) > 0$ on $\Sigma$. Then by the result of G\r{a}rding \cite{G59}, we have $\tilde{\kappa} \in \Gamma_k^{+}$.

It follows from (\ref{wmki}) that
\begin{equation}\label{cvhk}
	\chi(V)H_{k}(\tilde{\kappa})\int_{\Sigma}u d \mu = 	\int_{\Sigma}\chi(V) u H_k(\tilde{\kappa}) d \mu \geq \int_{\Sigma} \chi(V)(V-u) H_{k-1}(\tilde{\kappa}) d \mu.
\end{equation}
By the Newton-Maclaurin inequality (\ref{nm2}), we have
$$
H_{k-1}(\tilde{\kappa}) \geq H_{k}(\tilde{\kappa})^{\frac{k-1}{k}},
$$
thus
\begin{equation}\label{icvvu}
	\begin{aligned}
		\int_{\Sigma} \chi(V)(V-u) H_{k-1}(\tilde{\kappa}) d \mu & \geq  \int_{\Sigma} \chi(V)(V-u) H_{k}(\tilde{\kappa})^{\frac{k-1}{k}} d \mu  \\
		& = \left(\chi(V)H_{k}(\tilde{\kappa})\right)^{\frac{k-1}{k}} \int_{\Sigma} \chi(V)^{\frac{1}{k}}(V-u) d \mu.
	\end{aligned}
\end{equation}
Hence (\ref{cvhk}) and (\ref{icvvu}) imply
\begin{equation}\label{iu}
	\int_{\Sigma}u d \mu \geq \left(\chi(V)H_{k}(\tilde{\kappa})\right)^{-\frac{1}{k}} \int_{\Sigma} \chi(V)^{\frac{1}{k}}(V-u)  d \mu,
\end{equation}
and equality holds in and only if $\Sigma$ is a geodesic sphere. On the other hand, applying Proposition \ref{hkp} and the Newton-Maclaurin inequality (\ref{nm2}) we derive that
\begin{equation}\label{iu2}
	\begin{aligned}
		\int_{\Sigma}u d \mu & = (n+1)\int_{\Omega}V d \mathrm{vol} \leq \int_{\Sigma} \frac{V-u}{H_1(\tilde{\kappa})} d \mu \leq \int_{\Sigma} \frac{V-u}{H_k(\tilde{\kappa})^{\frac{1}{k}}} d \mu \\ & = \left(\chi(V)H_{k}(\tilde{\kappa})\right)^{-\frac{1}{k}} \int_{\Sigma} \chi(V)^{\frac{1}{k}}(V-u)  d \mu.
	\end{aligned}
\end{equation}
Finally combining (\ref{iu}) and (\ref{iu2}) together, we conclude that the equality holds in the Newton-MacLaurin inequality (\ref{nm2}), so $\Sigma$ is totally umbilical and then is a geodesic sphere. Moreover, suppose $\chi' > 0$. Since $\Sigma$ is totally umbilical and $\chi(V)H_{k}(\tilde{\kappa})$ is a constant, we have $V$ is a constant. Therefore the distance of each point in $\Sigma$ to the origin is a constant. So we conclude that $\Sigma$ is a centered geodesic sphere.

(ii) The first step is more or less the same as above. Lemma \ref{elliptic} and $\chi > 0$ imply that the ratio $\chi(V) \frac{H_k(\tilde{\kappa})}{H_l(\tilde{\kappa})}$ is a positive constant. Since by
the assumption $H_l(\tilde{\kappa})$ dose not vanish on $\Sigma$, $H_l(\tilde{\kappa})$ and $H_k(\tilde{\kappa})$ are positive everywhere in $\Sigma$. From \cite{G59}, we know that $\tilde{\kappa} \in \Gamma_k^{+}$.

If $l=0$, it is reduced to the case of Theorem \ref{thm1.1}(i). In the following, we consider the case $l \geq 1$. Denote the positive constant by $c$, namely,
$$
c := \chi(V) \frac{H_k(\tilde{\kappa})}{H_l(\tilde{\kappa})} > 0.
$$
Applying the Newton-Maclaurin inequality (\ref{nm1}), we note that
$$
\frac{H_k(\tilde{\kappa})}{H_{k-1}(\tilde{\kappa})} \leq \frac{H_l(\tilde{\kappa})}{H_{l-1}(\tilde{\kappa})}
$$
which yields
\begin{equation}\label{cformula1}
	\chi(V)\frac{H_{k-1}(\tilde{\kappa})}{H_{l-1}(\tilde{\kappa})} \geq c .
\end{equation}
It follows from (\ref{wmki}) and (\ref{mkf}) that
$$
\int_{\Sigma} \chi(V)(V-u) H_{k-1}(\tilde{\kappa}) d\mu \leq \int_{\Sigma}\chi(V) u H_k(\tilde{\kappa}) d \mu=c \int_{\Sigma} u H_l(\tilde{\kappa}) d \mu =c \int_{\Sigma}\left( V- u \right) H_{l-1}(\tilde{\kappa}) d \mu.
$$
This gives
$$
\int_{\Sigma} \left( V- u \right) \left(\chi(V) H_{k-1}(\tilde{\kappa})-c H_{l-1}(\tilde{\kappa})\right) d \mu\leq 0 .
$$
Note that $V-u>0$. The above together with (\ref{cformula1}) imply
$$
\chi(V)\frac{H_{k-1}(\tilde{\kappa})}{H_{l-1}(\tilde{\kappa})} = c,
$$
everywhere in $\Sigma$. By an iteration argument, we have
$$
\chi(V)\frac{H_{k-l}(\tilde{\kappa})}{H_{0}(\tilde{\kappa})}=\chi(V) H_{k-l}(\tilde{\kappa})=c,
$$
everywhere in $\Sigma$. Finally, from Theorem \ref{thm1.1}(i), we complete the proof.
\end{proof}
For the similar rigidity problem with weight $V-u$, the horoconvexity is necessary. That is the following theorem which is included in \cite[Proposition 8.1]{LX} and \cite[Theorem 1.3]{HWZ23}. For the completeness, we involve the proof here.
\begin{theorem}\label{urhc}
	Assume that $\chi(s)$ is a smooth, positive and monotone non-decreasing function defined on $\mathbb{R}^{+}$. Let $\Omega$ be a smooth bounded and uniformly h-convex $(\kappa_{i} \geq 1+\epsilon$, for some constant $\epsilon >0)$ domain in $\mathbb{H}^{n+1}$. If $\Sigma=\partial \Omega$ satisfies
	\begin{equation}
		\frac{H_k(\tilde{\kappa})}{H_l(\tilde{\kappa})}=\chi (V-u),
	\end{equation}
	where $0\leq l< k\leq n$ and $H_l(\tilde{\kappa})$ dose not vanish on $\Sigma$,
	then $\Sigma$ is a geodesic sphere. Moreover, if $\chi$ is strictly increasing, then $\Sigma$ is a centered geodesic sphere.
\end{theorem}
\begin{proof}
	As the hypersurface $\Sigma$ is uniformly h-convex, we have $\tilde{\kappa_{i}} >0$, $H_k(\tilde{\kappa}) >0$ and $\frac{\partial H_k(\tilde{\kappa})}{\partial \tilde{\kappa_{i}}} >0$. For the case of $l=0$, it has been proved in \cite{HWZ23}. In the following, we only focus on the case $l \geq 1$.
	
	Applying the Newton-Maclaurin inequality (\ref{nm1}), we note that
	$$
	\frac{H_k(\tilde{\kappa})}{H_{k-1}(\tilde{\kappa})} \leq \frac{H_l(\tilde{\kappa})}{H_{l-1}(\tilde{\kappa})}
	$$
	which implies
	\begin{equation}\label{cformula2}
		\frac{H_{k-1}(\tilde{\kappa})}{H_{l-1}(\tilde{\kappa})} \geq \chi (V-u) .
	\end{equation}
	Similar to the proof of (\ref{wmki}), one can prove that
	\begin{equation}\label{wmki1}
			\int_{\Sigma}\chi(V-u) u H_k(\tilde{\kappa}) d \mu \leq \int_{\Sigma} \chi(V-u)(V-u) H_{k-1}(\tilde{\kappa}) d \mu,
	\end{equation}
	where we used $\nabla_{i} (V-u)=-(h_i^j- \delta_i^j)\nabla_{j}V$. It follows from (\ref{wmki1}) and (\ref{mkf}) that
	$$
	\int_{\Sigma} \chi(V-u)(V-u) H_{l-1}(\tilde{\kappa}) d\mu \geq \int_{\Sigma}\chi(V-u) u H_l(\tilde{\kappa}) d \mu= \int_{\Sigma} u H_k(\tilde{\kappa}) d \mu = \int_{\Sigma}\left( V- u \right) H_{k-1}(\tilde{\kappa}) d \mu.
	$$
	This yields
	$$
	\int_{\Sigma} \left( V- u \right) \left(\chi(V-u) H_{l-1}(\tilde{\kappa})- H_{k-1}(\tilde{\kappa})\right) d \mu\geq 0 .
	$$
	Note that $V-u >0$. The above together with (\ref{cformula2}) imply
	$$
	\frac{H_{k-1}(\tilde{\kappa})}{H_{l-1}(\tilde{\kappa})}=\chi (V-u),
	$$
	everywhere in $\Sigma$. By an iteration argument, one obtains
	$$
	\frac{H_{k-l}(\tilde{\kappa})}{H_{0}(\tilde{\kappa})}= H_{k-l}(\tilde{\kappa})=\chi(V-u),
	$$
	everywhere in $\Sigma$. It reduces to the case of Hu-Wei-Zhou's result \cite{HWZ23}. We complete the proof.
\end{proof}

\section{Proof of Theorem \ref{thm1.3}}\label{sec4}
Next, we show the rigidity result for constant linear combinations of weighted shifted mean curvatures in the hyperbolic space. This argument needs pay more attention to the use of the Newton-Maclaurin inequality at the first step.
\begin{proof}[Proof of Theorem \ref{thm1.3}]
	(i) By Lemma \ref{elliptic} and non-vanishing of at least one coefficient, we know that $\sum_{i=1}^{l-1}a_i H_i (\tilde{\kappa}) > 0$. Since $\tilde{\kappa}\in\overline{\Gamma_{k}^+}$, we recall from (\ref{nm1}) that
	\begin{equation}\label{nmij}
		H_{i}(\tilde{\kappa})H_{j-1}(\tilde{\kappa}) \geq H_{i-1}(\tilde{\kappa})H_{j}(\tilde{\kappa}), \quad 1\leq i < j \leq k,
	\end{equation}
	where all equalities hold if and only if $\Sigma$ is umbilical. Multiplying (\ref{nmij}) by $a_i$, $b_j$ and $\chi$ and summing over $i$ and $j$, we obtain
	\begin{equation}\label{snmij}
		\sum_{i=1}^{l-1} a_i H_i (\tilde{\kappa}) \sum_{j=l}^k b_j(\chi(V)H_{j-1}(\tilde{\kappa}))  \geq \sum_{i=1}^{l-1} a_i H_{i-1}(\tilde{\kappa}) \sum_{j=l}^k b_j(\chi(V)H_{j}(\tilde{\kappa})).
	\end{equation}
	By the assumption
	$$
	\sum_{i=1}^{l-1} a_i H_i (\tilde{\kappa}) =\sum_{j=l}^k b_j (\chi(V)H_{j}(\tilde{\kappa})) >0, \quad 2 \leq l<k \leq n,
	$$
	we infer from (\ref{snmij}) that
	\begin{equation}\label{snmij1}
		\sum_{j=l}^k b_j (\chi(V)H_{j-1}(\tilde{\kappa})) \geq \sum_{i=1}^{l-1} a_i H_{i-1}(\tilde{\kappa}).
	\end{equation}
	On the other hand, we obtain from (\ref{wmki}) that
	\begin{equation}
		\begin{aligned}
				0 & =\int_{\Sigma} \left(\sum_{j=l}^k b_j (\chi(V)H_{j}(\tilde{\kappa})) -\sum_{i=1}^{l-1} a_i H_i (\tilde{\kappa}) \right) u d\mu \\ & \geq \int_{\Sigma}  \left(\sum_{j=l}^k b_j (\chi(V)H_{j-1}(\tilde{\kappa}))-\sum_{i=1}^{l-1} a_i H_{i-1} (\tilde{\kappa}) \right) (V-u) d\mu \geq 0.
		\end{aligned}
	\end{equation}
    Here, the last inequality follows from (\ref{snmij1}) and $V-u>0$. We conclude that the equality holds in the Newton-MacLaurin inequality (\ref{nm1}), which implies that $\Sigma$ is totally umbilical and thus a geodesic sphere.

    (ii) Making use of (\ref{nm1}) and (\ref{wmki}), we derive
    $$
    \begin{aligned}
    	a_{0} \int_{\Sigma} u d\mu & = \int_{\Sigma} u \left(\sum_{j=1}^{k}b_j \chi(V) H_j (\tilde{\kappa})\right) d\mu \geq \int_{\Sigma} (V-u) \left(\sum_{j=1}^{k}b_j \chi(V) H_{j-1} (\tilde{\kappa})\right) \frac{H_{1}(\tilde{\kappa})}{H_{1}(\tilde{\kappa})} d\mu \\
    	& \geq \int_{\Sigma} (V-u) \left(\sum_{j=1}^{k}b_j \chi(V) H_{j} (\tilde{\kappa})\right) \frac{1}{H_{1}(\tilde{\kappa})} d\mu = a_{0} \int_{\Sigma} \frac{V-u}{H_{1}(\tilde{\kappa})} d\mu \\
    	& \geq (n+1) a_{0}\int_{\Omega} V d\mathrm{vol} = a_{0} \int_{\Sigma} u d\mu,
    \end{aligned}
    $$
    where in the last inequality we used Proposition \ref{hkp}. Therefore, the equality in both case yields that $\Sigma$ is a geodesic sphere.

    (iii) The proof is essentially the same as above. One only needs to notice the slight difference regarding the value of indices. By Lemma \ref{elliptic} and non-vanishing of at least one coefficient, we have $\sum_{i=0}^{l-1}a_i H_i (\tilde{\kappa}) > 0$. Since $\tilde{\kappa}\in\overline{\Gamma_{k}^+}$, we recall from (\ref{nm1}) that
    \begin{equation}\label{nmij11}
    	H_{i}(\tilde{\kappa})H_{j+1}(\tilde{\kappa}) \leq H_{i+1}(\tilde{\kappa})H_{j}(\tilde{\kappa}), \quad 0\leq i < j \leq k,
    \end{equation}
    where all equalities hold if and only if $\Sigma$ is umbilical. Multiplying (\ref{nmij11}) by $a_i$, $b_j$ and $\chi$ and summing over $i$ and $j$, we get
    \begin{equation}\label{snmij11}
    	\sum_{i=0}^{l-1} a_i H_i (\tilde{\kappa}) \sum_{j=l}^k b_j (\chi(V)H_{j+1}(\tilde{\kappa})) \leq \sum_{i=0}^{l-1} a_i H_{i+1}(\tilde{\kappa}) \sum_{j=l}^k b_j (\chi(V)H_{j}(\tilde{\kappa})).
    \end{equation}
    Using the assumption
    $$
    \sum_{i=0}^{l-1} a_i H_i (\tilde{\kappa}) =\sum_{j=l}^k b_j (\chi(V)H_{j}(\tilde{\kappa})) >0, \quad 1 \leq l<k \leq n-1,
    $$
    we obtain from (\ref{snmij11}) that
    \begin{equation}\label{snmij111}
    	\sum_{j=l}^k b_j (\chi(V)H_{j+1}(\tilde{\kappa})) \leq \sum_{i=0}^{l-1} a_i H_{i+1}(\tilde{\kappa}).
    \end{equation}
    Applying (\ref{mkf}) and (\ref{wmki}) again,
    \begin{equation}
    	\begin{aligned}
    		0 & =\int_{\Sigma} \left(\sum_{i=0}^{l-1} a_i H_i (\tilde{\kappa})-\sum_{j=l}^k b_j (\chi(V)H_{j}(\tilde{\kappa})) \right) (V-u) d\mu \\ & \geq \int_{\Sigma}  \left(\sum_{i=0}^{l-1} a_i H_{i+1}(\tilde{\kappa})-\sum_{j=l}^k b_j H_{j+1}(\tilde{\kappa}) \right) u d\mu \geq 0.
    	\end{aligned}
    \end{equation}
    Here, the last inequality follows from (\ref{spf}) and (\ref{snmij111}). We finish the proof by examining the equality case as before.
\end{proof}

In a similar way, one can also prove the rigidity result for the weight $V-u$. We only state the result here and leave the proof to readers.
\begin{theorem}\label{cvu}
	Let $n\geq2$. Assume that $\chi(s)$ is a smooth, positive and monotone non-decreasing function defined on $\mathbb{R}^{+}$. Let $\Sigma$ be a closed uniformly h-convex hypersurface in $\mathbb{H}^{n+1}$. If either of the following holds:
	\begin{itemize}
		\item[(i)] $2\leq l<k\leq n$ and there are nonnegative constants $\{a_i\}_{i=1}^{l-1}$ and $\{b_j\}_{j=l}^{k}$, at least one of them not vanishing, such that
		\begin{equation*}
			\sum_{i=1}^{l-1}a_i (\chi(V-u)H_i (\tilde{\kappa}))=\sum_{j=l}^{k}b_j H_j (\tilde{\kappa});
		\end{equation*}
		\item[(ii)] $\Sigma$ is star-shaped, $1\leq l<k\leq n-1$ and there are nonnegative constants $\{a_i\}_{i=0}^{l-1}$ and $\{b_j\}_{j=l}^{k}$, at least one of them not vanishing, such that
		\begin{equation*}
			\sum_{i=0}^{l-1}a_i (\chi(V-u)H_i (\tilde{\kappa}))=\sum_{j=l}^{k}b_j H_j (\tilde{\kappa}),
		\end{equation*}
	\end{itemize}
	then $\Sigma$ is a geodesic sphere. Moreover, if $\chi$ is strictly increasing, then $\Sigma$ is a centered geodesic sphere.
\end{theorem}
\section{Proof of Theorems \ref{thm+1}--\ref{thm+3}}\label{sec+}
In this section, we use the similar idea as in the work of Kwong-Lee-Pyo \cite{KLP18}. The following formulas will play an essential role in the proof.
\begin{lemma}\label{pml}
	Let $\phi$ be a smooth function on a closed hypersurface $\Sigma$ in $\mathbb{H}^{n+1}$. We have
	\begin{equation}\label{pmk}
		\begin{aligned}
			\int_{\Sigma} \phi u H_k(\tilde{\kappa}) d \mu= & \int_{\Sigma} \phi(V-u) H_{k-1}(\tilde{\kappa}) d \mu  +  \frac{1}{k C_{n}^k} \int_{\Sigma} \left(T_{k-1}\right)_{i}^{j}(\tilde{h})\nabla^{i}\phi \nabla_j V d \mu ,
		\end{aligned}
	\end{equation}
	where $\widetilde{h_j^i}=h_j^i- \delta_j^i$.
\end{lemma}
\begin{proof}
	Similarly, by (\ref{tkv}), we arrive at
	\begin{equation*}
		\left(T_{k-1}\right)_i^j(\tilde{h}) \nabla^i \nabla_j V=kC_n^k \left(\left(V-u\right) H_{k-1}(\tilde{\kappa})-u H_k(\tilde{\kappa})\right).
	\end{equation*}
	Multiplying above equation by the function $\phi$ and integrating by parts, one obtains the desired result (\ref{pmk}).
\end{proof}

\begin{proof}[Proof of Theorem \ref{thm+1}]
	(i) It follows from (\ref{pmk}) and $\tilde{\kappa}\in\overline{\Gamma_{k}^+}$ that, for each $i$ and $j$,
	\begin{equation}\label{air}
		\begin{aligned}
			& \int_{\Sigma} a_i (r) \left((V-u)H_{i-1} (\tilde{\kappa}) - uH_i (\tilde{\kappa})\right) \\
			= & -\frac{1}{i C_{n}^i} \int_{\Sigma} \lambda(r) a_i' (r) \left(T_{i-1}\right)_{p}^{q}(\tilde{h})\nabla^{p}r \nabla_q r  \geq 0
		\end{aligned}
	\end{equation}
	and
	\begin{equation}\label{bjrv}
		\begin{aligned}
			& \int_{\Sigma} b_j (r) \left((V-u)H_{j-1} (\tilde{\kappa}) - uH_j (\tilde{\kappa})\right) \\
			= & -\frac{1}{j C_{n}^j} \int_{\Sigma} \lambda(r) b_j' (r) \left(T_{j-1}\right)_{p}^{q}(\tilde{h})\nabla^{p}r \nabla_q r  \leq 0
		\end{aligned}
	\end{equation}
	Summing (\ref{air}) over $i$ and (\ref{bjrv}) over $j$, and then taking the difference gives
	\begin{equation}\label{psnmij1}
		\begin{aligned}
			0 & =\int_{\Sigma} \left(\sum_{j=l}^k b_j (r) H_j (\tilde{\kappa})-\sum_{i=1}^{l-1} a_i (r) H_i (\tilde{\kappa}) \right) u d\mu \\ & \geq \int_{\Sigma}  \left(\sum_{j=l}^k b_j (r) H_{j-1} (\tilde{\kappa})-\sum_{i=1}^{l-1} a_i (r) H_{i-1} (\tilde{\kappa}) \right) (V-u) d\mu.
		\end{aligned}
	\end{equation}
	As in the proof of Theorem \ref{thm1.3}(i), one can obtain the following inequalities:
	\begin{equation}\label{bjainm}
		\sum_{j=l}^k b_j (r) H_{j-1}(\tilde{\kappa}) \geq \sum_{i=1}^{l-1} a_i (r) H_{i-1}(\tilde{\kappa}).
	\end{equation}
	Combining this with (\ref{psnmij1}), we conclude that all integrands in (\ref{psnmij1}) are zero. This implies (\ref{bjainm}) is an equality and hence $\Sigma$ is a geodesic sphere by the Newton-Maclaurin inequality.
	
	(ii) By dividing (\ref{a0r}) by $a_{0} (r)$, it suffices to prove the result in the case where
	\begin{equation}
		1 =\sum_{j=1}^{k}b_j (r) H_j (\tilde{\kappa}).
	\end{equation}
	From (\ref{nm1}) and (\ref{pmk}), we have
	$$
	\begin{aligned}
		 \int_{\Sigma} u d\mu & = \int_{\Sigma} u \left(\sum_{j=1}^{k}b_j (r) H_j (\tilde{\kappa})\right) d\mu \geq \int_{\Sigma} (V-u) \left(\sum_{j=1}^{k}b_j (r) H_{j-1} (\tilde{\kappa})\right) \frac{H_{1}(\tilde{\kappa})}{H_{1}(\tilde{\kappa})} d\mu \\
		& \geq \int_{\Sigma} (V-u) \left(\sum_{j=1}^{k}b_j (r) H_{j} (\tilde{\kappa})\right) \frac{1}{H_{1}(\tilde{\kappa})} d\mu = \int_{\Sigma} \frac{V-u}{H_{1}(\tilde{\kappa})} d\mu \\
		& \geq (n+1)\int_{\Omega} V d\mathrm{vol} = \int_{\Sigma} u d\mu
	\end{aligned}
	$$
	where in the last inequality we used Proposition \ref{hkp}. Therefore, the equality in both case yields that $\Sigma$ is a geodesic sphere.
	
	(iii) Similarly, as in the proof of Theorem \ref{thm1.3}(iii), one can obtain
	\begin{equation}\label{wsnmij111}
		\sum_{j=l}^k b_j (r) H_{j+1}(\tilde{\kappa}) \leq \sum_{i=0}^{l-1} a_i (r) H_{i+1}(\tilde{\kappa}).
	\end{equation}
	We infer from (\ref{pmk}) and $\tilde{\kappa}\in\overline{\Gamma_{k}^+}$ that, for each $i$ and $j$,
	\begin{equation}\label{air1}
		\begin{aligned}
			& \int_{\Sigma} a_i (r) \left((V-u)H_{i} (\tilde{\kappa}) - uH_{i+1} (\tilde{\kappa})\right) \\
			= & -\frac{1}{(i+1) C_{n}^{i+1}} \int_{\Sigma} \lambda(r) a_i' (r) \left(T_{i}\right)_{p}^{q}(\tilde{h})\nabla^{p}r \nabla_q r  \geq 0
		\end{aligned}
	\end{equation}
	and
	\begin{equation}\label{bjrv1}
		\begin{aligned}
			& \int_{\Sigma} b_j (r) \left((V-u)H_{j} (\tilde{\kappa}) - uH_{j+1} (\tilde{\kappa})\right) \\
			= & -\frac{1}{(j+1) C_{n}^{j+1}} \int_{\Sigma} \lambda(r) b_j' (r) \left(T_{j}\right)_{p}^{q}(\tilde{h})\nabla^{p}r \nabla_q r  \leq 0
		\end{aligned}
	\end{equation}
	Summing (\ref{air1}) over $i$ and (\ref{bjrv1}) over $j$, and then taking the difference gives
	\begin{equation}
		\begin{aligned}
			0 & =\int_{\Sigma} \left(\sum_{i=0}^{l-1} a_i (r) H_i (\tilde{\kappa})-\sum_{j=l}^k b_j (r) H_j (\tilde{\kappa}) \right) (V-u) d\mu \\ & \geq \int_{\Sigma}  \left(\sum_{i=0}^{l-1} a_i (r) H_{i+1}(\tilde{\kappa})-\sum_{j=l}^k b_j (r) H_{j+1}(\tilde{\kappa}) \right) u d\mu \geq 0.
		\end{aligned}
	\end{equation}
	where in the last inequality we used (\ref{spf}) and (\ref{wsnmij111}). By the Newton-Maclaurin inequality, it implies that $\Sigma$ is a geodesic sphere.
\end{proof}

\begin{proof}[Proof of Theorem \ref{thm+2}]
	By dividing (\ref{bjhj}) by $\eta (r)$, it suffices to prove the theorem in the case that
	\begin{equation}\label{bjhj1}
		\sum_{j=1}^{k}\bigg(a_j (r) H_j (\tilde{\kappa}) + b_j (r) H_1 (\tilde{\kappa}) H_{j-1} (\tilde{\kappa})\bigg) = 1.
	\end{equation}
	
	By $u\geq0$ and the Newton-Maclaurin inequality (\ref{nm1}), we have
	\begin{equation}\label{int u}
		\begin{aligned}
			\int_{\Sigma} u d\mu & = \int_{\Sigma} u \left(\sum_{j=1}^{k}\left(a_j (r) H_j (\tilde{\kappa})+ b_j (r) H_1 (\tilde{\kappa}) H_{j-1} (\tilde{\kappa})\right)\right) d\mu \\
			&=\int_{\Sigma}\sum_{j=1}^{k}a_j (r) u H_j (\tilde{\kappa}) d\mu + \int_{\Sigma}\sum_{j=1}^{k}b_j (r) u H_1 (\tilde{\kappa}) H_{j-1} (\tilde{\kappa}) d\mu \\
			&\geq \int_{\Sigma}\sum_{j=1}^{k}a_j (r) u H_j (\tilde{\kappa}) d\mu + \int_{\Sigma}\sum_{j=1}^{k}b_j (r) u H_j (\tilde{\kappa}) d\mu.
		\end{aligned}
	\end{equation}
	It follows from (\ref{pmk}) and $\tilde{\kappa}\in\overline{\Gamma_{k}^+}$ that
	\begin{equation}\label{bjr}
		\begin{aligned}
			& \int_{\Sigma} a_j (r) \left((V-u)H_{j-1} (\tilde{\kappa}) - uH_j (\tilde{\kappa})\right) \\
			= & -\frac{1}{j C_{n}^j} \int_{\Sigma} \lambda(r) a_j' (r) \left(T_{j-1}\right)_{p}^{q}(\tilde{h})\nabla^{p}r \nabla_q r  \leq 0.
		\end{aligned}
	\end{equation}
	Similarly,
	\begin{equation}\label{cjr}
		\begin{aligned}
			 &\int_{\Sigma} b_j (r) \left((V-u)H_{j-1} (\tilde{\kappa}) - uH_j (\tilde{\kappa}) \right) \\
			= & -\frac{1}{j C_{n}^j} \int_{\Sigma} \lambda(r) b_j' (r) \left(T_{j-1}\right)_{p}^{q}(\tilde{h})\nabla^{p}r \nabla_q r  \leq 0.
		\end{aligned}
	\end{equation}
	Substituting (\ref{bjr}) and (\ref{cjr}) into (\ref{int u}), using (\ref{nm1}), we have
	\begin{equation*}
		\begin{aligned}
			\int_{\Sigma} u d\mu & \geq \int_{\Sigma}\sum_{j=1}^{k}a_j (r) (V-u) H_{j-1} (\tilde{\kappa}) d\mu + \int_{\Sigma}\sum_{j=1}^{k}b_j (r) (V-u) H_{j-1} (\tilde{\kappa}) d\mu \\
			&\geq \int_{\Sigma}\sum_{j=1}^{k}a_j (r) (V-u) \frac{H_{j} (\tilde{\kappa})}{H_{1} (\tilde{\kappa})} d\mu + \int_{\Sigma}\sum_{j=1}^{k}b_j (r) (V-u) H_{j-1} (\tilde{\kappa}) d\mu \\
			&= \int_{\Sigma} \frac{V-u}{H_1 (\tilde{\kappa})}\left(\sum_{j=1}^{k}(a_j (r)H_{j} (\tilde{\kappa})+b_j (r)H_1 (\tilde{\kappa})H_{j-1} (\tilde{\kappa}))\right) d\mu = \int_{\Sigma} \frac{V-u}{H_1 (\tilde{\kappa})} d\mu.
		\end{aligned}
	\end{equation*}
	
	On the other hand, Hu-Wei-Zhou's inequality (Proposition \ref{hkp}) is the reverse inequality
	$$
	\int_{\Sigma} u d\mu \leq \int_{\Sigma} \frac{V-u}{H_1 (\tilde{\kappa})} d\mu.
	$$
	These two inequalities yield the equality in Hu-Wei-Zhou's inequality. We conclude that $\Sigma$ is a geodesic sphere.
\end{proof}
\begin{proof}[Proof of Theorem \ref{thm+3}]
	The assumption $\tilde{\kappa}\in\Gamma_{k}^+$ says $H_j (\tilde{\kappa})>0$ for all $j=0,\cdots,k$. It follows from (\ref{aij}) that $u>0$.
	
	Assume first that $k \geq 2$. By Newton-Maclaurin inequality (\ref{nm1}), we have for $0 \leq i < j \leq k$,
	\begin{equation}\label{gnm}
		\left(\frac{1}{H_1 (\tilde{\kappa})}\right)^{j-i} \leq \frac{H_i (\tilde{\kappa})}{H_j (\tilde{\kappa})}=\prod_{m=i}^{j-1}\frac{H_m (\tilde{\kappa})}{H_{m+1} (\tilde{\kappa})} \leq \left(\frac{H_{j-1} (\tilde{\kappa})}{H_j (\tilde{\kappa})}\right)^{j-i}.
	\end{equation}
	Therefore, by Newton-Maclaurin inequality (\ref{nm1}) again,
	\begin{equation}\label{u/v-u1}
		\beta \frac{u}{V-u}=\sum_{i < j} a_{i,j}\left(\frac{H_{i}(\tilde{\kappa})}{H_{j}(\tilde{\kappa})}\right)^{\frac{1}{j-i}} \leq \sum_{i < j} a_{i,j}\frac{H_{j-1} (\tilde{\kappa})}{H_j (\tilde{\kappa})}\leq \sum_{i < j} a_{i,j}\frac{H_{k-1} (\tilde{\kappa})}{H_k (\tilde{\kappa})}=\frac{H_{k-1} (\tilde{\kappa})}{H_k (\tilde{\kappa})},
	\end{equation}
	and
	\begin{equation}\label{u/v-u2}
		\beta \frac{u}{V-u}=\sum_{i < j} a_{i,j}\left(\frac{H_{i}(\tilde{\kappa})}{H_{j}(\tilde{\kappa})}\right)^{\frac{1}{j-i}} \geq \sum_{i < j} a_{i,j}\frac{1}{H_1 (\tilde{\kappa})}=\frac{1}{H_1 (\tilde{\kappa})}.
	\end{equation}
	The inequality (\ref{u/v-u1}) gives
	\begin{equation*}
		\beta \int_{\Sigma}uH_k (\tilde{\kappa})d\mu \leq \int_{\Sigma}(V-u)H_{k-1} (\tilde{\kappa})d\mu,
	\end{equation*}
	which in turn implies $\beta \leq 1$ by (\ref{mkf}).
	
	On the other hand, (\ref{u/v-u2}) yields
	\begin{equation*}
		\beta \int_{\Sigma}uH_1 (\tilde{\kappa})d\mu \geq \int_{\Sigma}(V-u)d\mu,
	\end{equation*}
	and thus $\beta \geq 1$ again by (\ref{mkf}).
	
	We conclude that $\beta = 1$ and all the inequalities in (\ref{gnm}) are all equalities. Therefore $\Sigma$ is umbilical and thus is a geodesic sphere.
	
	When $k=1$, (\ref{u/v-u2}) becomes an equality and hence $\beta = 1$ by (\ref{mkf}). By the Newton-Maclaurin inequality, we get
	\begin{equation}\label{gnm2}
		H_2 (\tilde{\kappa})\frac{u}{V-u}=\frac{H_{2}(\tilde{\kappa})}{H_{1}(\tilde{\kappa})} \leq \frac{H_{1}(\tilde{\kappa})}{H_{0}(\tilde{\kappa})}=H_{1}(\tilde{\kappa}).
	\end{equation}
	Integrating this inequality and comparing to the Minkowski type formula (\ref{mkf}) for $k=2$, we again infer that (\ref{gnm2}) is an equality and hence $\Sigma$ is a geodesic sphere.
\end{proof}

\section{Proof of Corollaries \ref{coro1} and \ref{coro2}}\label{sec5}
As in the proof of Lemma 2.4 in \cite{BCW21}, we first derive the formula for $H_{k}(\tilde{\kappa})$ in terms of $H_{i}(\kappa)$, $i=0,\cdots,k$.
\begin{lemma}\label{hkl}
	For a hypersurface $(\Sigma, g)$ in $\mathbb{H}^{n+1}$, its shifted mean curvature $H_{k}(\tilde{\kappa})$ of the induced metric $g$ can be expressed as follows
	\begin{equation}
		H_{k}(\tilde{\kappa}) = \sum_{i=0}^{k}(-1)^{k-i}\binom{k}{i}H_{i}(\kappa),
	\end{equation}
	where $\tilde{\kappa}=(\tilde{\kappa}_{1}, \cdots, \tilde{\kappa}_{n})=(\kappa_{1}-1,\cdots, \kappa_{n}-1)$ are the shifted principal curvatures of $\Sigma$.
\end{lemma}
\begin{proof}
	By the definition of the elementary symmetric polynomials, we have
	$$
	\prod_{i=1}^{n} (t+\tilde{\kappa}_{i}) = \sum_{k=0}^{n} \sigma_{k}(\tilde{\kappa})t^{n-k}.
	$$
	On the other hand,
	$$
	\begin{aligned}
		\prod_{i=1}^{n} (t+\tilde{\kappa}_{i}) & = \prod_{i=1}^{n} (t-1+\kappa_{i})=\sum_{l=0}^{n} \sigma_{l}(\kappa)(t-1)^{n-l} \\
		& = \sum_{l=0}^{n} \sigma_{l}(\kappa) \sum_{i=0}^{n-l}\binom{n-l}{i}t^{i} (-1)^{n-l-i} \\
		& = \sum_{k=0}^{n}\left(\sum_{i=0}^{k}\binom{n-i}{k-i}(-1)^{k-i}\sigma_{i}(\kappa)\right)t^{n-k}.
	\end{aligned}
	$$
	Comparing the coefficients of $t^{n-k}$, we have
	$$
	\begin{aligned}
		\sigma_{k}(\tilde{\kappa}) & = \sum_{i=0}^{k}\binom{n-i}{k-i}(-1)^{k-i}\sigma_{i}(\kappa) = \sum_{i=0}^{k}\binom{n-i}{k-i}(-1)^{k-i}\binom{n}{i}H_{i}(\kappa) \\
		& =\binom{n}{k}\sum_{i=0}^{k}(-1)^{k-i}\binom{k}{i}H_{i}(\kappa).
	\end{aligned}
	$$
	Thus
	$$
	H_{k}(\tilde{\kappa}) = \binom{n}{k}^{-1}\sigma_{k}(\tilde{\kappa})= \sum_{i=0}^{k}(-1)^{k-i}\binom{k}{i}H_{i}(\kappa).
	$$
\end{proof}
Next, we show that the Gauss-Bonnet curvature $L_{k}$ can be expressed as a linear combination of the shifted $m$th mean curvatures with $m$ ranging from $k$ to $2k$.
\begin{lemma}\label{lkl}
	For a hypersurface $(\Sigma, g)$ in $\mathbb{H}^{n+1}$, its Gauss-Bonnet curvature $L_k$ of the induced metric $g$ can be expressed as follows
	\begin{equation}
		L_{k} = \binom{n}{2k} (2k)! \sum_{j=0}^{k}2^{j}\binom{k}{j}H_{2k-j}(\tilde{\kappa}).
	\end{equation}
\end{lemma}
\begin{proof}
	The proof of this lemma can be found in \cite[Lemma 9.1]{HLW22}. For convenience of the readers, we include the proof here. Firstly, we recall the Gauss formula for hypersurfaces in hyperbolic space:
	\begin{equation}\label{rijsl}
		\begin{aligned}
			R_{i j}{ }^{s l} & =\left(h_i^s h_j^l-h_i^l h_j^s\right)-\left(\delta_i^s \delta_j^l-\delta_i^l \delta_j^s\right) \\
			&=\left(\tilde{h}_i^s+\delta_i^s\right)\left(\tilde{h}_j^l+\delta_j^l\right)-\left(\tilde{h}_i^l+\delta_i^l\right)\left(\tilde{h}_j^s+\delta_j^s\right)-\left(\delta_i^s \delta_j^l-\delta_i^l \delta_j^s\right) \\
			& =\left(\tilde{h}_i^s \tilde{h}_j^l+\delta_i^s \tilde{h}_j^l+\tilde{h}_i^s \delta_j^l\right)-\left(\tilde{h}_i^l \tilde{h}_j^s+\delta_i^l \tilde{h}_j^s+\tilde{h}_i^l \delta_j^s\right),
		\end{aligned}
	\end{equation}
	where $\tilde{h}_i^j=h_i^j-\delta_i^j$. Substituting (\ref{rijsl}) into the definition (\ref{lk}) of $L_k$ and noting (\ref{sgmk}), we have
	$$
	\begin{aligned}
		L_k & =\delta_{j_1 j_2 \cdots j_{2 k-1} j_{2 k}}^{i_1 i_2 \cdots i_{2 k-1} i_{2 k}}\left(\tilde{h}_{i_1}^{j_1} \tilde{h}_{i_2}^{j_2}+2 \tilde{h}_{i_1}^{j_1} \delta_{i_2}^{j_2}\right) \cdots\left(\tilde{h}_{i_{2 k-1}}^{j_{2 k-1}} \tilde{h}_{i_{2 k}}^{j_{2 k}}+2 \tilde{h}_{i_{2 k-1}}^{j_{2 k-1}} \delta_{i_{2 k}}^{j_{2 k}}\right) \\
		& =\sum_{m=0}^k \binom{k}{m}2^m \delta_{j_1 j_2 \cdots j_{2 k}}^{i_1 i_2 \cdots i_{2 k}} \delta_{i_2}^{j_2} \delta_{i_4}^{j_4} \cdots \delta_{i_{2 m}}^{j_{2 m}} \tilde{h}_{i_1}^{j_1} \tilde{h}_{i_3}^{j_3} \cdots \tilde{h}_{i_{2 m-1}}^{j_{2 m-1}} \tilde{h}_{i_{2 m+1}}^{j_{2 m+1}} \tilde{h}_{i_{2 m+2}}^{j_{2 m+2}} \cdots \tilde{h}_{i_{2 k-1}}^{j_{2 k-1}} \tilde{h}_{i_{2 k}}^{j_{2 k}} \\
		& =\sum_{m=0}^k\binom{k}{m} 2^m(n+1-2 k) \cdots(n+1-2 k+m-1)(2 k-m) !\binom{n}{2k-m} H_{2 k-m}(\tilde{\kappa}) \\
		& =\binom{n}{2k}(2 k) ! \sum_{m=0}^k\binom{k}{m} 2^m H_{2 k-m}(\tilde{\kappa}) .
	\end{aligned}
	$$
	Here in the first equality we used the symmetry of the generalized Kronecker delta, and in the third equality we used the basic property of the generalized Kronecker delta
	$$
	\delta_{j_1 j_2 \cdots j_{p-1} j_p}^{i_1 i_2 \cdots i_{p-1} i_p} \delta_{i_1}^{j_1}=(n+1-p) \delta_{j_2 j_3 \cdots j_p}^{i_2 i_3 \cdots i_p} .
	$$
\end{proof}
From above Lemmas and Theorem \ref{thm+1}(ii) with constant coefficients, the proof of Corollaries \ref{coro1} and \ref{coro2} are apparent.
\vspace{5mm}

{\bf Acknowledgements.}
Sheng was partially supported by National Key R$\&$D Program of China(No. 2022YFA1005500) and  Natural Science Foundation of China under Grant No. 12031017 and No. 11971424. Wu was partially supported by National Key R$\&$D Program of China(No. 2022YFA1005501) and  Natural Science Foundation of China under Grant No. 11731001.

\end{document}